\documentclass[a4paper,11pt]{article}

\usepackage{tikz}
\usepackage{pgf}
\usepackage{enumitem}
\usepackage{amscd}

\usepackage[T1]{fontenc}
\usepackage{lmodern}

\usepackage{dsfont}

\usepackage[latin1]{inputenc}
\usepackage{amsmath}
\usepackage{amsthm}
\usepackage{amssymb}
\usepackage{mathrsfs}
\usepackage{graphicx}
\usepackage[all]{xy}
\usepackage{hyperref}

\usepackage{makeidx}

\usepackage{stmaryrd}
\usepackage{caption}

\usepackage{abstract} 

\newtheorem{thm}{Theorem}[section]
\newtheorem{cor}[thm]{Corollary}
\newtheorem{claim}[thm]{Claim}
\newtheorem{fact}[thm]{Fact}

\newtheorem{lemma}[thm]{Lemma}
\newtheorem{prop}[thm]{Proposition}

\theoremstyle{definition}
\newtheorem{definition}[thm]{Definition}
\newtheorem{ex}[thm]{Example}

\newtheorem{question}[thm]{Question}

\newcommand{\QM}{\mathrm{QM}(\Gamma, \mathcal{G})}

\def\rquotient#1#2{%
	\makeatletter
	\raise.3ex\hbox{$#1$}/\lower.3ex\hbox{$#2$}%
	\makeatother
}	

\makeatletter
\newcommand{\subjclass}[2][2010]{%
	\let\@oldtitle\@title%
	\gdef\@title{\@oldtitle\footnotetext{#1 \emph{Mathematics subject classification.} #2}}%
}
\newcommand{\keywords}[1]{%
	\let\@@oldtitle\@title%
	\gdef\@title{\@@oldtitle\footnotetext{\emph{Key words and phrases.} #1.}}%
}
\makeatother

\newcommand{\Address}{{% additional braces for segregating \footnotesize
		\bigskip
		\small
		
		\textsc{Institut Montpellierain Alexander Grothendieck, 499-554 Rue du Truel, 34090 Montpellier, France.}\par\nopagebreak
		\textit{E-mail address}: \texttt{anthony.genevois@umontpellier.fr}
\medskip

		\textsc{Institut de Math\'ematiques de Jussieu-Paris Rive Gauche, Place Aur\'elie Nemours, 75013 Paris, France.}\par\nopagebreak
		\textit{E-mail address}: \texttt{romain.tessera@imj-prg.fr}
\medskip
		
}}

\makeindex

\title{A note on morphisms to wreath products}
\date{\today}
\author{Anthony Genevois and Romain Tessera}


\begin{document}

\maketitle

\begin{abstract}
Given a morphism $\varphi : G \to A \wr B$ from a finitely presented group $G$ to a wreath product $A \wr B$, we show that, if the image of $\varphi$ is a sufficiently large subgroup, then $\mathrm{ker}(\varphi)$ contains a non-abelian free subgroup and $\varphi$ factors through an acylindrically hyperbolic quotient of $G$. As direct applications, we classify the finitely presented subgroups in $A \wr B$ up to isomorphism and we deduce that a group having a wreath product $(\text{non-trivial}) \wr (\text{infinite})$ as a quotient must be SQ-universal (extending theorems of Baumslag and Cornulier-Kar). Finally, we exploit our theorem in order to describe the structure of the automorphism groups of several families of wreath products, highlighting an interesting connection with the Kaplansky conjecture on units in group rings. 
\end{abstract}

\tableofcontents

\section{Introduction}

\noindent
The starting point of the article is the following surprising observation, which can be found in Baumslag's book \cite[Theorem~IV.7]{MR1243634}: a finitely presented group $G$ that admits a lamplighter group $\mathbb{Z}/n\mathbb{Z} \wr \mathbb{Z}$ ($n \geq 2$) as a quotient is necessarily \emph{large}, i.e.\ it contains a finite-index subgroup that surjects onto a free group of rank two. The result is surprising, at least at first glance, because, under the seemingly innocent assumption that our group $G$ is finitely presented, we deduce from the existence of a specific solvable quotient the fact that $G$ is very far from being solvable. Unfortunately, the proof given in \cite{MR1243634} does not really clarify this phenomenon. A well-known alternative argument (see for instance \cite[Lemma~3.2]{MR1958994}), which, in our opinion, is more transparent, is the following. It turns out that, even though $\mathbb{Z}/n\mathbb{Z} \wr \mathbb{Z}$ is solvable, it can be approximated by virtually free groups: by truncating the infinite presentation
$$\langle t, a_i \ (i \in \mathbb{Z}) \mid a_i^n=1, ta_it^{-1}=a_{i+1} \ (i \in \mathbb{Z}), [a_i,a_j]=1 \ (i,j \in \mathbb{Z}) \rangle$$
of $\mathbb{Z}/n\mathbb{Z} \wr \mathbb{Z}$ as
$$\langle t, a_i \ (i \in \mathbb{Z}) \mid a_i^n=1, ta_it^{-1}=a_{i+1} \ (i \in \mathbb{Z}), [a_i,a_j]=1 \ (|i-j|\leq k) \rangle$$
for some $k \geq 1$, one recognises the HNN extension
$$\left\langle t,a_0, \ldots, a_k \left| \begin{array}{c} a_i^n=1 \ (0 \leq i \leq k) \\ \left[ a_i ,a_j \right]=1 \ (0 \leq i,j \leq k) \\ ta_it^{-1}=a_{i+1} \ (0 \leq i \leq k-1) \end{array} \right. \right\rangle = \langle a_0,\ldots, a_k \rangle \underset{\langle a_0, \ldots, a_{k-1} \rangle}{\ast}$$
of $(\mathbb{Z}/n\mathbb{Z})^{k+1}$, which is virtually free non-abelian.
Then, it is not difficult to deduce from the fact that $G$ is finitely presented that a quotient map $G \twoheadrightarrow \mathbb{Z}/n \mathbb{Z} \wr \mathbb{Z}$ factors through one of these truncations (see Fact \ref{fact:Algebra}), proving that $G$ is indeed large. 

\medskip \noindent
This point of view was developed in \cite{MR2764930} for arbitrary wreath products $A \wr B : = \bigoplus_B A \rtimes B$, by observing that truncations of a well-chosen (infinite) presentation have a very specific algebraic structure: they decompose as a semidirect product between a graph product of copies of $A$ with $B$. In particular, this observation allowed the authors of \cite{MR2764930} to generalise Baumslag's observation (see Theorem~\ref{thm:CK} below). 

\medskip \noindent
In our article, we combine this point of view on wreath products together with the geometric perspective on graph products introduced by the first-named author in \cite{Qm}. As a result, we are able to deduce severe restrictions on morphisms from finitely presented groups to wreath products. The next statement, which is the main result of the article, summarises the restrictions we find:

\begin{thm}\label{thm:Main}
Let $A,B,G$ be three groups. Assume that $G$ is finitely presented. Every morphism $\varphi : G \to A \wr B$ factors through a quotient $\overline{G}$ of $G$ such that:
\begin{itemize}
	\item if $\varphi(G)$ has an infinite projection on $B$ and intersects non-trivially $\bigoplus_B A \leq A \wr B$, then $\mathrm{ker}(\varphi)$ contains a non-abelian free subgroup and $\overline{G}$ is acylindrically hyperbolic;
	\item if $\varphi(G)$ has infinite projection on $B$ and is not contained in a conjugate of $B$, then $\overline{G}$ acts on a finite-dimensional CAT(0) cube complex with unbounded orbits and with hyperplane-stabilisers virtually embedded in finite powers of $A$.
\end{itemize}
In particular, in the second item, $\overline{G}$ semisplits over a subgroup that virtually embeds in a finite power of $A$; so, if $A$ is finite, then $\overline{G}$ is multi-ended.
\end{thm}

\noindent
Recall that a finitely generated group $A$ \emph{semisplits} over a subgroup $B$ if the Schreier graphs of $A$ with respect to $B$ (constructed from any finite generating set of $A$) is multi-ended (which amounts to saying that the relative number of ends $e(A,B)$ is at least two). The connection between semisplittings and actions on CAT(0) cube complexes is explained in \cite{MR1347406}. In particular, groups acting on CAT(0) cube complexes semisplit over stabilisers of \emph{essential} hyperplanes. 

\medskip \noindent
Loosely speaking, this theorem states that either $\varphi : G \to A \wr B$ is uninteresting, in the sense that it essentially comes from a morphism to $B$ or a power of $A$; or it has a large kernel and it factors through a negatively curved quotient of $G$ that acts non-trivially on a finite-dimensional CAT(0) cube complex. 

\medskip \noindent
Theorem~\ref{thm:Main} might be difficult to digest, so let us illustrate a few specific aspects of it.

\paragraph{Nearly injective morphisms.} As a direct consequence of Theorem \ref{thm:Main}, a morphism $\varphi : G \to A \wr B$ from a finitely presented group $G$ to a wreath product $A \wr B$ that has a \emph{small} kernel (i.e.\ that does not contain a non-abelian free subgroup) must be quite specific: either the image $\varphi(G)$ of $G$ projects isomorphically into $B$, or it lies in $\bigoplus_B A \rtimes  F \leq A \wr B$ for some finite subgroup $F \leq B$. Applying this observation to injective morphisms allows us to classify the finitely presented subgroups in $A \wr B$.

\begin{cor}\label{cor:FPSub}
Let $A,B,G$ be three groups. Assume that $G$ is finitely presented and $B$ infinite. Then $G$ is isomorphic to a subgroup of $A \wr B$ if and only if $G$ is isomorphic to a subgroup of $B$ or if there exists an $n \geq 1$ such that $G$ is a (subgroup in $A^n$)-by-(finite subgroup in $B$). 
\end{cor}

\noindent
In other words, the finitely presented subgroups of a wreath product $A \wr B$ essentially comes from the finitely presented subgroups of $B$ and the powers of $A$. The wreath product structure does not create new finitely presented subgroups.

\begin{proof}[Proof of Corollary \ref{cor:FPSub}.]
Assume that $G$ is isomorphic to a subgroup of $A \wr B$. By applying Theorem \ref{thm:Main} to an injective morphism $\varphi : G \hookrightarrow A \wr B$, it follows that either $G$ is isomorphic to a finitely presented subgroup of $B$ or it embeds in $\bigoplus_B A \rtimes F$ for some finite subgroup $F \leq B$. In the latter case, because $G$ is finitely generated, only finitely many copies of $A$ contains non-trivial coordinates of elements in the image of $G$, so $G$ must be a (subgroup in $A^n$)-by-(finite subgroup in $B$) for some $n \geq 1$. Conversely, it is clear that finitely presented subgroups of $B$ are finitely presented subgroups of $A \wr B$. According to \cite{MR49892}, a group $G$ that is a (subgroup in $A^n$)-by-(finite subgroup in $B$) embeds into $A^n \wr F$ for some finite subgroup $F \leq B$. But $A^n \wr F$ embeds into $A \wr B$. Indeed, because $B$ is infinite, we can fix some $X \subset B$ which is the union of $n$ pairwise distinct orbits under the action $F \curvearrowright B$ by left multiplications. Then the subgroup $\bigoplus_X A \rtimes F \leq \bigoplus_BA \rtimes B$ is isomorphic to $A^n \wr F$. Therefore, $G$ embeds into $A \wr B$, as desired.
\end{proof}

\noindent
For instance, for all $n,m \geq 1$, the finitely presented subgroups of $\mathbb{Z}^n \wr \mathbb{Z}^m$ are free abelian; and the finitely presented subgroups of $\mathbb{Z}/n\mathbb{Z} \wr \mathbb{Z}^m$ are free abelian or sums of cyclic groups whose orders divide $n$.

\paragraph{Nearly surjective morphisms.} In the opposite direction, if a finitely presented group $G$ admits a morphism $\varphi : G \to A \wr B$ to a wreath product $A \wr B$ that has a ``sufficiently large'' image (i.e.\ such that $\varphi(G)$ has a non-trivial intersection with $\bigoplus_B A$ and an infinite projection on $B$) then Theorem \ref{thm:Main} implies that $G$ must have an acylindrically hyperbolic quotient \cite{OsinAcyl}. As a consequence, $G$ must be \emph{SQ-universal} (i.e.\ every countable group is isomorphic to a subgroup in a quotient of $G$). Thus, one obtains the following statement, in the same spirit as Baumslag's observation:

\begin{cor}\label{cor:SQ}
Let $A,B$ be two groups and $G$ a finitely presented groups. Assume that $A$ is non-trivial and that $B$ is infinite. If $G$ admits $A \wr B$ as a quotient, then it must be SQ-universal.
\end{cor}

\begin{proof}
The corollary follows from Theorem \ref{thm:Main} and from the fact that acylindrically hyperbolic groups are SQ-universal \cite[Theorem 2.33]{DGO}.
\end{proof}

\noindent
Observe that, compared to Baumslag's result, we weakened the conclusion: instead of being large, our group is only SQ-universal. In fact, deducing that the group has to be large can be done for a vast class of wreath products:

\begin{thm}[\cite{MR2764930}]\label{thm:CK}
Let $A$ be a group admitting a non-trivial finite quotient and $B$ an infinite residually finite group. Every finitely presented group admitting $A \wr B$ as a quotient must be large.
\end{thm}

\noindent
Nevertheless, one easily sees that it is not true that a finitely presented group admitting a wreath product $(\text{non-trivial})\wr(\text{infinite})$ as a quotient is automatically large, as shown by the next examples.

\begin{ex}\label{ex:Simple}
If $A,B$ are two finitely presented groups with no proper finite-index subgroup, then the free product $A \ast B$ defines a finitely presented group that admits the wreath product $A \wr B$ as a quotient but that is not large. Indeed, every finite-index subgroup of $A \ast B$ has to contain both $A$ and $B$, and so cannot be proper since $A$ and $B$ generate $A \ast B$. In other words, $A \ast B$ also has no proper finite-index subgroup. On the other hand, every morphism $A \ast B \to \mathbb{F}_2$ must contain $A$ and $B$ in its kernel, so its image must be trivial. In conclusion, $A \ast B$ is far from being large.

\medskip \noindent
As a consequence, the free product $\mathbb{Z}/2\mathbb{Z} \ast A$, which admits $\mathbb{Z}/2\mathbb{Z} \wr A$ as a quotient, cannot be large either. Indeed, the kernel of the projection $\mathbb{Z}/2\mathbb{Z} \ast A \twoheadrightarrow \mathbb{Z}/2\mathbb{Z}$ is $A \ast zAz^{-1}$, where $z$ denotes the non-trivial element of the factor $\mathbb{Z}/2\mathbb{Z}$. But our previous observation shows that $A \ast zAz^{-1}$ cannot be large, which implies that $\mathbb{Z}/2\mathbb{Z} \ast A$ is not large, as desired.

\medskip \noindent
One the simplest finitely presented group with no proper finite-index subgroup is the so-called Higman group $H:= \langle x_i \ (i \in \mathbb{Z}/4\mathbb{Z}) \mid x_{i+1}x_ix_{i+1}^{-1}=x_i^2 \ (i \in \mathbb{Z}/4\mathbb{Z}) \rangle$ \cite{MR38348}. Other possible examples come from infinite finitely presented simple groups, such as Thompson's groups $T$ and $V$ \cite{MR1426438}. 
\end{ex}

\noindent
Example~\ref{ex:Simple} justifies that the property of being large has been replaced with the property of being SQ-universal. 

\medskip \noindent
As a by-product, observe that Corollary~\ref{cor:SQ}, combined with Pichot's observation from \cite{MR3786300}, can be used in order to show that some (free) Burnside groups are not finitely presented. Even though this is known for the Burnside groups for which we know that they are infinite (see for instance the proof of \cite[Theorem~7.7]{MR3211906}), up to our knowledge there is no proof that they cannot be finitely presented as soon as they are infinite. 

\begin{prop}\label{prop:Burnside}
Let $k,m,n \geq 1$ be three integers. If the Burnside group $B(m,n)$ is infinite, then $B(m+1,kn)$ is not finitely presented. 
\end{prop}

\noindent
Recall that $B(m,n)$ is defined as $\langle x_1, \ldots, x_m \mid w(x_1, \ldots, x_m)^n=1 \text{ for every word } w \rangle$.

\begin{proof}[Proof of Proposition~\ref{prop:Burnside}.]
The wreath product $\mathbb{Z}/k\mathbb{Z} \wr B(m,n)$ admits a generating set of cardinality $m+1$ and all its elements have exponent $kn$. Therefore, there exists a surjective morphism $B(m+1,kn) \twoheadrightarrow \mathbb{Z}/k\mathbb{Z} \wr B(m,n)$. If $B(m,n)$ is infinite, then it follows from Corollary \ref{cor:SQ} that $B(m+1,kn)$ cannot be finitely presented. (Observe that the latter assertion also follows from Corollary~\ref{cor:FW} below.)
\end{proof}

\paragraph{Fixed-point property.} Finally, it follows from Theorem~\ref{thm:Main} that finitely presented groups satisfying strong fixed-point properties do not have many morphisms to wreath products. More precisely:

\begin{cor}\label{cor:FW}
Let $A,B,G$ be three groups and $\varphi : G \to A \wr B$ a morphism. If $G$ is finitely presented and satisfies the property $(\mathrm{FW}_\mathrm{fin})$, then $\varphi(G)$ has a finite projection on $B$ or it lies in a conjugate of $B$. 
\end{cor}

\noindent
We say that a group satisfies the property $(\mathrm{FW}_\mathrm{fin})$ if it cannot act on a finite-dimensional CAT(0) cube complex with unbounded orbits, or, equivalently, without a global fixed point. Because Kazhdan's property (T) can be formulated as a fixed-point property on (complete) median spaces \cite{medianviewpoint} and that median graphs, a discrete counterpart of median spaces, define canonically the same objects as CAT(0) cube complexes \cite{mediangraphs, Roller}, one can think of the property $(\mathrm{FW}_\mathrm{fin})$ as being a discrete and finite-dimensional version of Kazhdan's property~(T). Examples include groups satisfying (T), such as $\mathrm{SL}(n,\mathbb{Z})$ for $n \geq 3$ and uniform lattices in quaternionic hyperbolic spaces, but also finitely generated torsion groups \cite{MR1347406} and Thompson's groups $T,V$ \cite{MR3918481, MR4028980}.

\paragraph{Automorphisms.} However, our main application of Theorem~\ref{thm:Main} deals with automorphisms of wreath products. It seems that automorphism groups of (finitely generated) wreath products have not been extensively studied in the literature. Here, we observe that, if $B$ from our wreath product $A \wr B$ is finitely presented and one-ended, then Theorem~\ref{thm:Main} immediately implies that every automorphism of $A \wr B$ sends $B$ to one of its conjugates. We deduce from this rigidity the following structure result:

\begin{cor}\label{cor:IntroAuto}
Let $A$ be a finite cyclic group and $B$ a one-ended finitely presented group. Then 
$$\mathrm{Aut}(A \wr B)=\left( \bigoplus_B A \oplus A[B]^\times \right)\rtimes \mathrm{Aut}(B).$$
\end{cor}

\noindent
Here, $\bigoplus_B A$ is identified with its image under the canonical morphism $A \wr B \to \mathrm{Inn}(A \wr B)$ (which is injective here) to the subgroup of inner automorphisms, $\mathrm{Aut}(B)$ is the subgroup of $\mathrm{Aut}(A \wr B)$ corresponding the automorphisms of $B$ that we naturally extend to $A \wr B$, and $A[B]^\times$ corresponds to a subgroup of $\mathrm{Aut}(A \wr B)$ isomorphic to the units of the group ring $A[B]$. We refer to Section~\ref{section:Auto} for more details. 

%\medskip \noindent
%As a particular case, $\mathrm{Out}(A \wr \mathbb{Z}^n)$ contains a finite-index subgroup isomorphic to $\mathrm{GL}(n,A)$ for every $n \geq 2$.  \color{red}J'imagine que $\mathrm{GL}(n,A)$ devrait etre $\mathrm{GL}(n,\mathbb{Z})$ ici... Mais je ne comprends pas l'enonce: que devient $A[B]^\times$ quand on passe a Out? \color{black} This contrasts with the fact that $\mathrm{Out}(A \wr \mathbb{Z})$ is not finitely generated when $A$ is a finite cyclic group whose order is divisible by a square. (See Proposition~\ref{prop:AutoCyclic}.) 

\medskip \noindent
Interestingly, Corollary~\ref{cor:IntroAuto} exhibits a connection between automorphism groups of wreath products and units in group rings. Understanding the structure of group rings, or even units in group rings, is notably difficult. As an illutration, recall that Kaplansky's unit conjecture - which claims that the units in the group ring $F[G]$ of a torsion-free group $G$ over a field $F$ are all of the form $k\cdot g$ for some $k \in K^\ast$, $g \in G$ - has been disproved only recently \cite{Kaplansky}. By extension, fully understanding the structure of the automorphism groups of specific wreath products might be quite difficult.

\medskip \noindent
As an immediate consequence of Corollary~\ref{cor:IntroAuto}:

\begin{cor}
Let $A$ be a finite cyclic group and $B$ a one-ended finitely presented group. If $\mathrm{Aut}(B)$ is finitely generated and if $B$ satisfies the Kaplansky unit conjecture over $A$, then $\mathrm{Out}(A \wr B)$ is finitely generated. 
\end{cor}

\noindent
Here, we say that \emph{$B$ satisfies the Kaplansky unit conjecture over $A$} if every unit in $A[B]$ is trivial, i.e. of the form $a \cdot b$ for some $a \in A^\times$, $b \in B$. For instance, when $A$ is a field, it is known that $B$ satisfies the Kaplansky unit conjecture over $A$ if it satisfies the so-called \emph{unique product property}, which includes in particular orderable \cite{MR798076} and locally indicable groups \cite{MR310046}. 

\medskip \noindent
Section~\ref{section:Auto} contains a discussion about automorphism groups of wreath products, including a generalisation of Corollary~\ref{cor:IntroAuto} (see Proposition~\ref{prop:AutFreeProduct}) and some hints suggesting that the structure of automorphism groups of arbitrary wreath products can be quite tricky.

\paragraph{Acknowledgments.} We are grateful to R\'emi Coulon and Simon Andr\'e for instructive discussions related to Propositions~\ref{prop:Burnside} and~\ref{prop:Logic}; and to Yves Cornulier for his comments that allow us to improve the presentation of the article. We also thank Francesco Fournier-Facio for relevant comments regarding Section~\ref{section:Auto}, and the anonymous referee for their comments which help us to clarify the exposition of the article.

\section{A few words about graph products of groups}\label{section:GP}

\noindent
As shown in the next section, wreath products turn out to be tightly related to \emph{graph products of groups}. This section is dedicated to basic definitions and preliminary statements related to these products. 

\medskip \noindent
Given a simplicial graph $\Gamma$ and a collection of groups $\mathcal{G}= \{G_u \mid u \in V(\Gamma) \}$ indexed by the vertices of $\Gamma$, the \emph{graph product} $\Gamma \mathcal{G}$ is defined by relative presentation
$$\langle G_u \ (u \in V(\Gamma)) \mid [G_u,G_v]=1 \ (\{u,v\} \in E(\Gamma)) \rangle$$
where $E(\Gamma)$ denotes the edge-set of $\Gamma$ and where $[G_u,G_v]=1$ is a shortcut for $[x,y]=1$ for all $x \in G_u$, $y \in G_v$. The groups in $\mathcal{G}$ are referred to as \emph{vertex-groups}. If all the vertex-groups are isomorphic to a single group $A$, we also denote $\Gamma \mathcal{G}$ by $\Gamma A$.

\paragraph{Normal form.}\index{Normal form of graph products} Fix a simplicial graph $\Gamma$ and a collection of groups $\mathcal{G}$ indexed by $V(\Gamma)$. A \emph{word} in $\Gamma \mathcal{G}$ is a product $g_1 \cdots g_n$ where $n \geq 0$ and where each $g_i$ belongs to a vertex-group; the $g_i$'s are the \emph{syllables} of the word, and $n$ is the \emph{length} of the word. Clearly, the following operations on a word do not modify the element of $\Gamma \mathcal{G}$ it represents:
\begin{itemize}
	\item[(O1)] delete the syllable $g_i=1$;
	\item[(O2)] if $g_i,g_{i+1} \in G$ for some $G \in \mathcal{G}$, replace the two syllables $g_i$ and $g_{i+1}$ by the single syllable $g_ig_{i+1} \in G$;
	\item[(O3)] if $g_i$ and $g_{i+1}$ belong to two adjacent vertex-groups, switch them.
\end{itemize}
A word is \emph{graphically reduced} if its length cannot be shortened by applying these elementary moves. Given a word $g_1 \cdots g_n$ and some $1\leq i<n$, if the vertex-group associated to $g_i$ is adjacent to each of the vertex-groups of $g_{i+1}, \ldots, g_n$, then the words $g_1 \cdots g_n$ and $g_1\cdots g_{i-1} \cdot g_{i+1} \cdots g_n \cdot g_i$ represent the same element of $\Gamma \mathcal{G}$; we say that $g_i$ \textit{shuffles to the right}. Analogously, one can define the notion of a syllable shuffling to the left. If $g=g_1 \cdots g_n$ is a graphically reduced word and $h$ a syllable, then a graphical reduction of the product $gh$ is given by
\begin{itemize}
	\item $g_1 \cdots g_n$ if $h=1$;
	\item $g_1 \cdots g_{i-1} \cdot g_{i+1} \cdots g_n$ if $h \neq 1$, $g_i$ shuffles to the right and $g_i=h^{-1}$;
	\item $g_1 \cdots g_{i-1} \cdot g_{i+1} \cdots g_n \cdot (g_ih)$ if $h \neq 1$, $g_i$ and $h$ belong to the same vertex-group, $g_i$ shuffles to the right, and $g_i \neq h^{-1}$, where $(g_ih)$ is thought of as a single syllable;
	\item $g_1 \cdots g_nh$ otherwise. 
\end{itemize}
As a consequence, every element $g$ of $\Gamma \mathcal{G}$ can be represented by a graphically reduced word, and this word is unique up to applying the operation (O3). Often, elements of $\Gamma \mathcal{G}$ and graphically reduced words are identified. 
We refer to \cite{GreenGP} for more details (see also \cite{GPvanKampen} for a more geometric approach). 

\medskip \noindent
In graph products, there is also a notion of cyclic reduction. A word $g_1 \cdots g_n$ is \emph{graphically cyclically reduced} if it is graphically reduced and if there do not exist two indices $1 \leq i < j \leq n$ such that $g_i$ and $g_j$ belong to the same vertex-group, $g_i$ shuffles to the left, and $g_j$ shuffles to the right. Notice that, as proved in \cite{GreenGP}, every element $g \in \Gamma \mathcal{G}$ can be written as a graphically reduced product $xyx^{-1}$ where $y$ is graphically cyclically reduced; moreover, $y$ is unique up to cyclic permutations of its syllables. The subgraph of $\Gamma$ generated by the vertices corresponding to the vertex-groups used to write the syllables of $y$ is the \emph{essential support} of $g$, denoted by $\mathrm{supp}(g)$. In other words, the essential support of an element $g$ is the smallest induced subgraph $\Lambda \subset \Gamma$ such that $g$ belongs to a conjugate of the subgroup generated by the vertex-groups indexing the vertices of $\Lambda$.

\begin{figure}
\begin{center}
\includegraphics[scale=0.4]{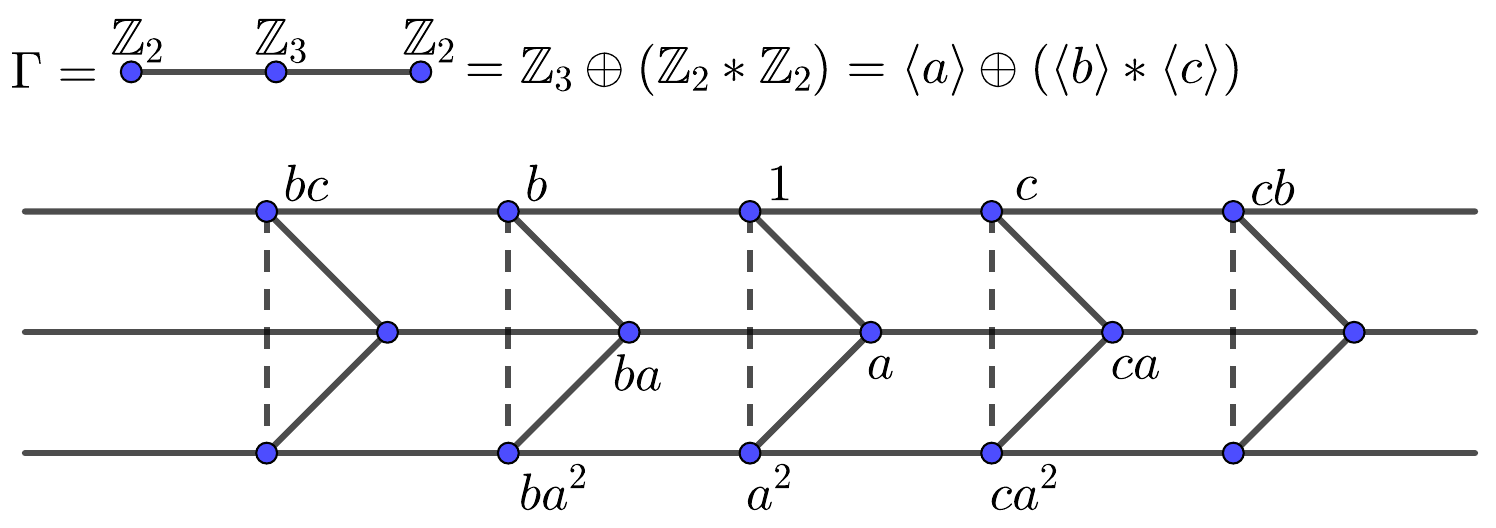}
\caption{Example of a graph $\QM$.}
\label{Cayl}
\end{center}
\end{figure}

\paragraph{Quasi-median geometry.} Let $\Gamma$ be a simplicial graph and $\mathcal{G}$ a collection of groups labelled by $V(\Gamma)$. The geometry of the graph product $\Gamma \mathcal{G}$ turns out to be encoded in the following Cayley graph:
$$\QM : = \mathrm{Cayl} \left( \Gamma \mathcal{G}, \bigcup\limits_{u \in V(\Gamma)} G_u \backslash \{1 \} \right),$$
i.e.\ the graph whose vertices are the elements of the group $\Gamma \mathcal{G}$ and whose edges link two distinct vertices $x,y \in \Gamma \mathcal{G}$ if $y^{-1}x$ is a non-trivial element of some vertex-group. Like in any Cayley graph, edges of $\QM$ are labelled by generators, namely by elements of vertex-groups. By extension, paths in $\QM$ are naturally labelled by words of generators. In particular, geodesics in $\QM$ correspond to words of minimal length. More precisely:

\begin{lemma}\label{lem:geodesicsQM}
Let $\Gamma$ be a simplicial graph and $\mathcal{G}$ a collection of groups indexed by $V(\Gamma)$. 
Fix two vertices $g,h \in \QM$. If $s_1 \cdots s_n$ is a graphically reduced word representing $g^{-1}h$, then 
$$g, \ gs_1, \ gs_1s_2, \ldots, gs_1 \cdots s_{n-1}, \ gs_1 \cdots s_{n-1}s_n = h$$
defines a geodesic in $\QM$ from $g$ to $h$. Conversely, if $s_1, \ldots, s_n$ is the sequence of elements of $\Gamma \mathcal{G}$ labelling the edges of a geodesic in $\QM$ from $g$ to $h$, then $s_1 \cdots s_n$ is a graphically reduced word representing $g^{-1}h$. 
\end{lemma}

\noindent
In \cite{Qm}, it is shown that $\QM$ is a \emph{quasi-median graph} and a general geometry of such graphs is developed in analogy with CAT(0) cube complexes. The central idea is that hyperplanes can be defined in quasi-median graphs in such a way that the geometry of the quasi-median graph reduces to the combinatorics of its hyperplanes. Because quasi-median graphs are understood through the behaviour of their hyperplanes, we do not define quasi-median graphs here and refer to \cite{quasimedian} for various equivalent characterisations. We only describe hyperplanes. 

\begin{definition}\label{def:hyp}
Let $X$ be a quasi-median graph. A \emph{hyperplane} is a class of edges with respect to the transitive closure of the relation claiming that two edges that are opposite in a square or that belong to a common triangle are equivalent. The \emph{carrier} of a hyperplane $J$, denoted by $N(J)$, is the smallest induced subgraph of $\QM$ containing $J$. Two hyperplanes $J_1$ and $J_2$ are \emph{transverse} if they intersect a square along two distinct pairs of opposite edges.
\end{definition}
\begin{figure}
\begin{center}
\includegraphics[trim={0 16.5cm 10cm 0},clip,scale=0.4]{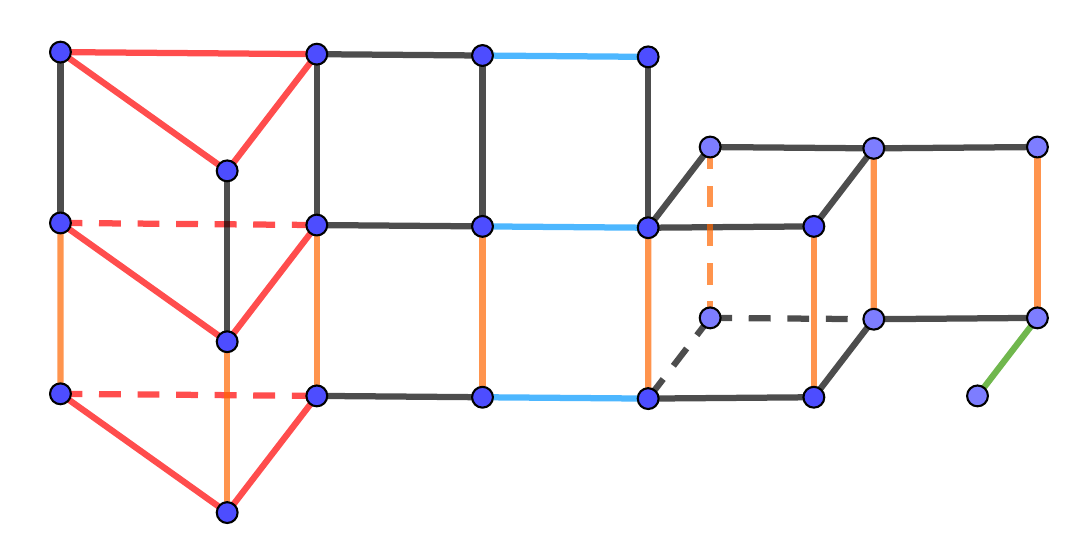}
\caption{Four hyperplanes in a quasi-median graph, coloured in red, blue, green, and orange. The orange hyperplane is transverse to the blue and red hyperplanes.}
\label{figure3}
\end{center}
\end{figure}

\noindent
We refer to Figure \ref{figure3} for examples of hyperplanes. The main difference between quasi-median graphs and CAT(0) cube complexes is that hyperplanes separate the complex into exactly two connected components in the latter case when they separate the graph in at least two (and possibly infinitely many) connected components in the former case. 

\begin{thm}\emph{\cite[Proposition 2.15]{Qm}}
Let $X$ be a quasi-median graph and $J$ a hyperplane. The graph obtained from $X$ by removing the interiors of all the edges in $J$ contains at least two connected components, referred to as \emph{sectors}. 
\end{thm}

\paragraph{Negative curvature.} Typically, graph products are acylindrically hyperbolic. More precisely, given a finite simplicial graph $\Gamma$ and a collection of non-trivial groups $\mathcal{G}$ indexed by $V(\Gamma)$, the graph product $\Gamma \mathcal{G}$ is acylindrically hyperbolic as soon as the graph $\Gamma$ contains at least two vertices and is not a join \cite{MinasyanOsin}. (Recall that $\Gamma$ is a join if there exists a non-trivial partition $V(\Gamma) = A \sqcup B$ such that every vertex in $A$ is adjacent to every vertex in $B$.) This negative curvature comes from the \emph{irreducible elements} in $\Gamma \mathcal{G}$, namely the elements whose essential support is neither a single vertex nor contained in a join. Such elements act on the quasi-median graph $\QM$ like rank-one isometries, as justified by the following statement (which is a consequence of \cite[Proposition 4.24]{AutG}).

\begin{lemma}\label{lem:skewer}
Let $\Gamma$ be a simplicial graph and $\mathcal{G}$ a collection of groups indexed by $V(\Gamma)$. If $g \in \Gamma \mathcal{G}$ is an irreducible element, then it \emph{skewers} a pair of \emph{strongly separated} hyperplanes in $\mathrm{QM}(\Gamma, \mathcal{G})$. 
\end{lemma}

\noindent
Recall that, given a quasi-median graph $X$, an isometry $g \in \mathrm{Isom}(X)$ \emph{skewers} a pair of hyperplanes $\{J_1,J_2\}$ if there exist a power $n \in \mathbb{Z}$ and two sectors $J_1^+,J_2^+$ respectively delimited by $J_1,J_2$ such that $g^n \cdot J_1^+ \subset J_2^+$. The hyperplanes $J_1$ and $J_2$ are \emph{strongly separated} if no hyperplane in $X$ is transverse to both $J_1$ and $J_2$. 

\medskip \noindent
Our proof of Theorem \ref{thm:Main} will be based on the following criterion:

\begin{prop}\label{prop:AcylHyp}
Let $G$ be a group acting on a quasi-median graph $X$. If $G$ contains an element that skewers a pair of strongly separated hyperplanes $\{J_1,J_2\}$ such that $\mathrm{stab}(J_1) \cap \mathrm{stab}(J_2)$ is finite, then $G$ must be acylindrically hyperbolic or virtually cyclic.
\end{prop}

\begin{proof}
Consider the collection of walls 
$$\mathcal{SW}:= \{ \{S,S^c\} \mid S \subset X \text{ sector}\}.$$
According to \cite[Lemma 4.17]{Qm}, any two vertices in $X$ are separated by only finitely many walls in $\mathcal{SW}$, i.e.\ $(X, \mathcal{SW})$ is a wallspace. If $G$ contains an element that skewers a pair $\{J_1,J_2\}$ of strongly separated hyperplanes in $X$, then by construction such an element also skewers a pair $\{J_1',J_2'\}$ of strongly separated hyperplanes in the CAT(0) cube complex $C(X, \mathcal{SW})$ constructed by cubulating the wallspace $(X, \mathcal{SW})$, where $J_1'$ (resp. $J_2'$) is the hyperplane in $C(X, \mathcal{SW})$ associated to the wall $\{S_1,S_1^c\}$ (resp. $\{S_2,S_2^c\}$) defined by the sector $S_1$ (resp. $S_2$) delimited by $J_1$ (resp. $J_2$) that contains $J_2$ (resp. $J_1$). As $\mathrm{stab}(J_1) \cap \mathrm{stab}(J_2) = \mathrm{stab}(J_1') \cap \mathrm{stab}(J_2')$, the desired conclusion follows from \cite[Theorem 2]{MR4028832}. 
\end{proof}

\section{Proof of the main theorem}\label{section:Proof}

\noindent
Fix two groups $A$ and $B$. The wreath product $A \wr B$ admits the following relative presentation:
$$\langle A_b \ (b \in B), B \mid bA_1b^{-1}=A_b \ (b \in B), [A_1,A_b]=1 \ (b \in B) \rangle$$
where $[A_1,A_b]=1$ is a shortcut for $[x,y]=1$ for all $x \in A_1$, $y \in A_b$. Fixing a subset $S \subset B$, let $A \square_S B$ denote the group defined by the truncated presentation
$$\langle A_b \ (b \in B), B \mid bA_1b^{-1}=A_b \ (b \in B), [A_1,A_b]=1 \ (b \in S) \rangle,$$
or equivalently by 
$$\langle A_b \ (b \in B), B \mid bA_1b^{-1}=A_b \ (b \in B), [A_{b_1},A_{b_2}]=1 \ (b_1,b_2 \in B, b_1^{-1}b_2 \in S) \rangle.$$
We denote by $\pi_S$ the canonical quotient map $A \square_S B \twoheadrightarrow A \wr B$. From the latter presentation, it follows that $A \square_S B$ decomposes as the semidirect product $\Gamma_S A \rtimes B$ where $\Gamma_S:= \mathrm{Cayl}(B,S)$ and where $B$ acts on $\Gamma_SA$ by permuting the vertex-groups according to its action on the Cayley graph $\Gamma_S$ by left-multiplication. 

\medskip \noindent
The key idea, first expressed in \cite{MR2295543}, is that the wreath product $A \wr B$ can be approximated by $A \square_SB=\Gamma_S A \rtimes B$ by taking a sufficiently large subset $S \subset B$. As an illustration of this idea:

\begin{lemma}\label{lem:Approx}
Let $A,B$ be two groups. For every finitely presented group $G$ and every morphism $\varphi : G \to A \wr B$, there exists a finite subset $S \subset B$, and a morphism $\psi : G \to A \square_S B$ such that $\varphi = \pi_S \circ \psi$. 
\end{lemma}

\begin{proof}
This is an easy consequence of a classical fact about truncated presentations; see for instance \cite[Lemma~3.1]{MR2764930}. We give a short argument here for the reader's convenience. Fix a generating set $\{s_1, \ldots, s_k\}$ of $G$. For every $1 \leq i \leq k$, let $w_i$ denote a word written over $A \cup B$ such that $\varphi(s_i)=w_i$ in $A \wr B$. Such a word exists because $A \cup B$ generates $A \wr B$. Thought of as words written over $\{w_1, \ldots, w_k\}^{\pm 1}$, the relations of a finite presentation of $G$ can be deduced from finitely many relations of the presentation of $A \wr B$ written above. Consequently, by taking a sufficiently large finite subset $S \subset B$, thanks to von Dyck's theorem on group presentations we can define a morphism $\psi : G \to A \square_S B$ satisfying $\varphi = \pi_S \circ \psi$ by setting $\psi(s_i)=w_i$ in $A \square_SB$. 
\end{proof}

\noindent
We are now ready to prove the main theorem of the article.

\begin{proof}[Proof of Theorem \ref{thm:Main}.]
Let $\varphi : G \rightarrow A \wr B$ be a morphism. Because $G$ is finitely presented, we know from Lemma \ref{lem:Approx} that there exist a finite subset $S \subset B$ and a morphism $\psi : G \to A \square_S B$ such that $\varphi = \pi_S \circ \psi$. We claim that $\psi(G)$ is the quotient of $G$ we are looking for.

\medskip \noindent
Consider the action of $\Gamma_SA$ on its quasi-median graph $\mathrm{QM}$. This action clearly extends to $\Gamma_S A \rtimes B$ since the action of $B$ on $\Gamma_SA$ by conjugations permutes the vertex-groups (and so stabilises the generating set used to define $\mathrm{QM}$). To be precise, for all $g \in \Gamma_S$ and $b \in B$, we have
$$gb : x \mapsto gbxb^{-1}, \text{ for every vertex $x \in \mathrm{QM}$}.$$
Our goal is to deduce from its action on $\mathrm{QM}$ that $\psi(G)$ is acylindrically hyperbolic, which will prove our second item.

\begin{claim}\label{claim:one}
The group $\Gamma_S A \rtimes B$ acts on $\mathrm{QM}$ with trivial stabilisers of oriented edges. 
\end{claim}
%\color{red} It seems that we consider in the proof edge-fixators. This proves that edge-stabilizers are of order $1$ or $2$. Order $2$ is possible: if $s$ is of order $2$, and commutes with $x$ (e.g.\ $x=s$). \color{black}
\noindent
Indeed, if $e=(x,xs)$ is an edge of $\mathrm{QM}$ and if $h \in \Gamma_SA \rtimes B$ stabilises $e$, then, writing $h=gb$ for some $g \in \Gamma_SA$ and $b \in B$, we have
$$\left\{ \begin{array}{l} gbxb^{-1} = x \\ gbxsb^{-1}=xs \end{array} \right.$$
From
$$x \cdot bsb^{-1}= gbxb^{-1} \cdot bsb^{-1}= gbxsb^{-1}=xs,$$
we deduce that $bsb^{-1}=s$. But $B$ acts freely on the vertices of $\Gamma_S$, so necessarily $b=1$. It then follows that $g=1$ from the equality $gbxb^{-1}=x$. Claim~\ref{claim:one} is proved.

\medskip \noindent
In order to prove the first item from our theorem, we assume that $\varphi(G)$ has an infinite projection on $B$ and a non-trivial intersection with $\bigoplus_B A \leq A \wr B$. As a consequence, $\psi(G)$ has an infinite projection on $B$ and a non-trivial intersection with $\Gamma_SA$. We want to prove that $\mathrm{ker}(\varphi)$ contains a non-abelian free subgroup and that $\psi(G)$ is acylindrically hyperbolic.

\begin{claim}\label{claim:three}
The intersection $\psi(G) \cap \Gamma_SA$ contains an irreducible element and is not finitely generated.
\end{claim}

\noindent
Fix an enumeration $\{g_1,g_2, \ldots\}$ of $\psi(G) \cap \Gamma_SA$. For every $n \geq 1$, set $G_n:= \langle g_1, \ldots, g_n \rangle$ and let $\Lambda_n$ denote the smallest subgraph in $\Gamma_S$ such that $G_n$ lies in a conjugate of the subgroup $\langle \Lambda_n \rangle$ generated by the vertex-groups indexed by the vertices in $\Lambda_n$ (such a subgraph exists according to \cite[Proposition 3.10]{MR3365774} and it is finite). Because $B$ acts metrically properly on $\Gamma_S$ and because $\psi(G)$ has an infinite projection on $B$, the sequence $\Lambda_1 \subset \Lambda_2 \subset \cdots$ does not stabilise. This implies that $\psi(G) \cap \Gamma_SA$ cannot be finitely generated, as claimed. Also, since $\Gamma_S$ is locally finite, there must exist an index $n \geq 1$ such that $\Lambda_n$ contains at least two vertices and is not contained in a join. According to \cite[Theorem 6.16]{MinasyanOsin}, $G_n$ contains an element whose essential support is $\Lambda_n$. Such an element is irreducible by construction, concluding the proof of our claim. 

\medskip \noindent
By combining Claim \ref{claim:three} and Lemma \ref{lem:skewer}, it follows that $\psi(G)$ contains an element that skewers a pair $\{J_1,J_2\}$ of strongly separated hyperplanes in $\mathrm{QM}$. Let $x_1 \in N(J_1)$ and $x_2 \in N(J_2)$ be two vertices minimising the distance between the carriers of $J_1$ and $J_2$. Because $J_1$ and $J_2$ are strongly separated, it follows from \cite[Corollary~2.16]{GM} that $x_1$ and $x_2$ are the only such vertices. Consequently, $\mathrm{stab}(J_1) \cap \mathrm{stab}(J_2) \subset \mathrm{stab}(x_1) \cap \mathrm{stab}(x_2)$. Because there exist only finitely many geodesics between $x_1$ and $x_2$ (see Lemma~\ref{lem:geodesicsQM}), we deduce from Claim \ref{claim:one} that $\mathrm{stab}(x_1) \cap \mathrm{stab}(x_2)$, and a fortiori $\mathrm{stab}(J_1) \cap \mathrm{stab}(J_2)$, is finite. As a consequence of Proposition \ref{prop:AcylHyp}, $\psi(G)$ is either acylindrically hyperbolic or virtually cyclic. But we know from Claim~\ref{claim:three} that $\psi(G)$ contains a subgroup that is not finitely generated, so the latter case cannot happen.

\medskip \noindent
Thus, we have proved that $\psi(G)$ is acylindrically hyperbolic. Now we want to prove that $\mathrm{ker}(\varphi)$ contains a non-abelian free subgroup. Because $\psi(G)$ is acylindrically hyperbolic, so are its infinite normal subgroup \cite[Lemma~7.2]{OsinAcyl}. Therefore, if $\psi(G) \cap \mathrm{ker}(\pi_S)$ is infinite, then it has to contain a non-abelian free subgroup, whose pre-image under $\psi$ defines a non-abelian free subgroup in $\mathrm{ker}(\varphi)$. Otherwise, if $\psi(G) \cap \mathrm{ker}(\varphi)$ is finite, there exists a finite-index subgroup $H \leq G$ such that $\pi_S$ is injective on $\psi(H)$. Consequently, $\varphi(H)$ defines an acylindrically hyperbolic group in $A \wr B$. Moreover, the fact that $\varphi(G)$ intersects non-trivially $\bigoplus_B A$ and has an infinite projection on $B$, implies that $\varphi(G) \cap \bigoplus_BA$ is actually infinite, from which we deduce that $\varphi(H) \cap \bigoplus_BA$ must be infinite as well. As an infinite normal subgroup of acylindrically hyperbolic group, $\varphi(H) \cap \bigoplus_BA$ must be acylindrically hyperbolic. In particular, there must exist some $h \in \varphi(H) \cap \bigoplus_BA$ whose centraliser in $\varphi(H)$ is virtually cyclic (for instance a generalised loxodromic element \cite[Lemma~6.8]{OsinAcyl}). Because $\varphi(H)$ has an infinite projection on $B$, there must exist $a \in \bigoplus_BA$ and $b \in B$ such that $ab$ belongs to $\varphi(H)$ and such that $\mathrm{supp}(h) \cap b \cdot \mathrm{supp}(h)=\emptyset$, where the support of an element $g$ of $\bigoplus_BA$ refers to the smallest $S \subset B$ such that $g \in \bigoplus_SA$. Because
$$\mathrm{supp} \left( h^{ab} \right) = b \cdot \mathrm{supp} \left( h^a \right) = b \cdot \mathrm{supp}(h),$$
we have $\langle h,h^{ab} \rangle = \langle h \rangle \oplus \langle h^{ab}\rangle \simeq \mathbb{Z}^2$, contradicting the fact that $h$ has a virtually cyclic centraliser in $\varphi(H)$. This concludes the proof of the first item of our theorem. 

\medskip \noindent
Finally, in order to justify the second item from our theorem, we need to show that, if $\varphi(G)$ has infinite projection on $B$ and is not contained in a conjugate of $B$, then $\psi(G)$ acts non-trivially on a finite-dimensional CAT(0) cube complex with hyperplane-stabilisers that virtually embed in finite powers of $A$. Because $\mathrm{QM}$ is canonically (hence $\psi(G)$-equivariantly) quasi-isometric to a CAT(0) cube complex of dimension $\mathrm{clique}(\Gamma_S)$ according to \cite[Proposition~4.16 and Corollary~8.7]{Qm}, it suffices to show that, under our assumptions, $\psi(G)$ has unbounded orbits in $\mathrm{QM}$ and that hyperplane-stabilisers in $A \square_S B$ virtually embed into finite powers of $A$. The latter assertion is clear thanks to \cite[Corollary~8.10]{Qm} (see also \cite[Theorem~2.10]{GM}). If $\psi(G)$ instead has finite orbits, then it has to stabilise a \emph{prism} of $\mathrm{QM}$, namely a coset $g\langle \Lambda \rangle \subset \mathrm{QM}$ for some $g \in \Gamma_S A$ and some complete subgraph $\Lambda \subset \Gamma_S$ \cite[Theorem~2.115 and Corollary~8.7]{Qm}. Up to conjugating $\psi(G)$ with an element in $\bigoplus_A B$, we assume that $g=1$. Observe that, for all $a \in \Gamma_S A$ and $b \in B$, we have $ab \cdot \langle \Lambda \rangle = a \langle b \Lambda \rangle$; consequently, $ab$ stabilises $\langle \Lambda \rangle$ if and only if $b\Lambda= \Lambda$ and $a \in \langle \Lambda \rangle$. Thus, either $\Lambda$ is empty and the stabiliser of $\langle \Lambda \rangle$ in $A \square_S B$ is $B$; or $\Lambda$ is non-empty and the stabiliser of $\langle \Lambda \rangle$ in $A \square_S B$ is $\langle \Lambda \rangle \cdot \mathrm{stab}_B(\Lambda)$, where $\mathrm{stab}_B(\Lambda)$ is necessarily finite. In the former case, we have 
$$\varphi(G) = \pi_S(\psi(G)) \subset \pi_S(B)=B,$$
proving that $\varphi(G)$ is contained in $B$ (up to conjugacy); and, in the latter case, we have 
$$\varphi(G) = \pi_S(\psi(G)) \underset{\text{virtually}}{\subset} \pi_S( \Gamma_S A) = \bigoplus\limits_{B} A,$$
proving that $\varphi(G)$ has finite projection on $B$. 
%Observe that, if $\Lambda$ is empty, then $\langle \Lambda \rangle$ is reduced to the single vertex $1$ and its stabiliser in $A \square_S B$ is $B$. If so, 
%$$\varphi(G) = \pi_S(\psi(G)) \subset \pi_S(B)=B,$$
%proving that $\varphi(G)$ is contained in $B$ (up to conjugacy). Otherwise, if $\Lambda$ is non-empty, then the stabiliser of $\langle \Lambda \rangle$ in $\Gamma_S A$ has finite index in the stabiliser of $\langle \Lambda \rangle$ in $A \square_S B$: \color{blue}indeed, the stabiliser of $\langle \Lambda \rangle$ in $A \square_S B$ has a finite-index subgroup  preserving the product structure \color{red} Do you mean that it preserves each factor? In any case why do we care? \color{blue} of $\langle \Lambda \rangle$ (which is, graphically, a product $|\Lambda|$ complete graphs); the subgroup $\langle \Lambda \rangle \leq \Gamma_S A$ acts on the subgraph $\langle \Lambda \rangle \leq \mathrm{QM}$ with a single orbit of vertices and $|\Lambda|$ orbits of edges (one for each factor of the prism $\langle \Lambda \rangle$); edge-stabilisers in $A \square_S B$ are trivial according to Claim~\ref{claim:one} \color{red}This justification is cryptic. Could you unzip it a little?  \color{blue} 
%$$\varphi(G) = \pi_S(\psi(G)) \underset{\text{virtually}}{\subset} \pi_S( \Gamma_S A) = \bigoplus\limits_{B} A,$$
%proving that $\varphi(G)$ has finite projection on $B$. 

\medskip \noindent
This concludes the proof of the second item from our theorem. In order to prove the last assertion of our theorem, we deduce from \cite[Theorem~5.1]{MR1347406} that $\psi(G)$ semisplits over some hyperplane-stabiliser, so the conclusion follows from what we already know.
\end{proof}

\section{Automorphisms}\label{section:Auto}

\noindent
In this section, we record some consequences of Theorem~\ref{thm:Main} on automorphisms of wreath products. Our starting point is given by the following observation:

\begin{lemma}\label{lem:AutoConjB}
Let $A,B$ be two groups with $A$ finite and $B$ finitely presented and one-ended. Every automorphism of $A \wr B$ sends $B$ to a conjugate of $B$.
\end{lemma}

\begin{proof}
Let $\varphi : A \wr B \to A \wr B$ be an automorphism. Because $B$ is finitely presented, Theorem\ref{thm:Main} applies to the restriction of $\varphi$ to $B$. Since $B$ is one-ended and $\varphi$ is injective, it follows that $\varphi(B)$ is contained in a conjugate of $B$, say $gBg^{-1}$. Similarly, $\varphi^{-1}(B)$ must be contained in some conjugate $hBh^{-1}$. We have
$$\varphi(B) \subset gBg^{-1} = g \varphi ( \varphi^{-1}(B)) g^{-1}= g \varphi(h) \varphi(B) \varphi(h)^{-1}g^{-1}.$$
But $B$ (and a fortiori $\varphi(B)$) is almost malnormal in $A \wr B$ (see for instance \cite[Lemma~7.11]{GTwreath}), so the inclusion above must be an equality since $B$ is infinite, hence $\varphi(B)=gBg^{-1}$. 
\end{proof}

\noindent
Fix two groups $F,H$. In our definitions of automorphisms below, we refer to an element of $F \wr H$ as a pair $(c,p)$ where $c : H \to F$ is a finitely supported map and where $p \in H$. Consequently, $(c_1,p_1) \cdot (c_2,p_2)= (c_1(\cdot) c_2(p_1^{-1} \cdot), p_1p_2)$ for all $(c_1,p_1),(c_2,p_2) \in F \wr H$. 
\begin{description}
	\item[Automorphisms of $H$.] Given an automorphism $\varphi \in \mathrm{Aut}(H)$, the map $$\overline{\varphi} : \left\{ \begin{array}{ccc} F \wr H & \to & F \wr H \\ (c,p) & \mapsto & \left( c \circ \varphi^{-1}, \varphi(p) \right) \end{array} \right.$$ defines an automorphism of $F \wr H$. Moreover, the map $\varphi \mapsto \overline{\varphi}$ defines a monomorphism $\mathrm{Aut}(H) \hookrightarrow \mathrm{Aut}(F \wr H)$. In the sequel, we will identify $\mathrm{Aut}(H)$ with its image in $\mathrm{Aut}(F \wr H)$ when this does not cause ambiguity. 
\end{description}
\begin{description}
	\item[$H$-equivariant automorphisms of $\bigoplus_H F$.] For every automorphism $\varsigma \in \mathrm{Aut}(\bigoplus_H F)$ that is equivariant with respect to the action of $H$ on $\bigoplus_H F$ by coordinate-shifting, i.e.\ $\varsigma(h \cdot c)= h \cdot \varsigma(c)$ for all $h \in H$ and $c \in \bigoplus_H F$, the map $$\overline{\varsigma} : \left\{ \begin{array}{ccc} F \wr H & \to & F \wr H \\ (c,p) & \mapsto & (\varsigma(c),p) \end{array} \right.$$ defines an automorphism of $F \wr H$. We denote by $\mathrm{Aut}_H(\bigoplus_H F)$ the subgroup of $\mathrm{Aut}(F \wr H)$ consisting of all these automorphisms.
\end{description}
Verifying that these maps define automorphisms is a direct and straightforward computation. Notice that, for every automorphism $\varphi \in \mathrm{Aut}(F)$,
$$\widehat{\varphi} : (c,h) \mapsto (\varphi \circ c,h), \ (c,h) \in F \wr H$$
yields an example of an $H$-equivariant automorphism of $\bigoplus_H F$. When there is no possible confusion, we identify $\mathrm{Aut}(F)$ with its image in $\mathrm{Aut}(F \wr H)$. 

\medskip \noindent
Our first observation is that these automorphisms, together with the inner automorphisms, are sufficient to generate the automorphism group in many cases.

\begin{prop}\label{prop:AutOneEnded}
Let $F$ be a finite group and $H$ a one-ended finitely presented group. Then 
$$\mathrm{Aut}(F \wr H) = \mathrm{Inn}(F \wr H) \cdot \left(\mathrm{Aut}_H \left( \bigoplus\limits_H F \right)\rtimes \mathrm{Aut}(H)\right).$$
In particular, $\mathrm{Aut}(F \wr H)$ is generated by the inner automorphisms, the automorphisms of $H$, and the $H$-equivariant automorphisms of $\bigoplus_H F$. 
\end{prop}

\begin{proof}
Let $\varphi \in \mathrm{Aut}(F \wr H)$ be an automorphism. As a consequence of Lemma~\ref{lem:AutoConjB}, $\varphi(H)$ is a conjugate of $H$. Up to composing $\varphi$ with an inner automorphism, we assume that $\varphi(H)=H$. Up to composing $\varphi$ with an automorphism of $H$, we also assume that $\varphi$ fixes $H$ pointwise. Because $\bigoplus_HF$ is characteristic \cite[Theorem~9.12]{MR188280}, it follows that there exists an automorphism $\varsigma : \bigoplus_H F \to \bigoplus_H F$ such that $\varphi(f,p)=(\varsigma(f),p)$ for every $(f,p) \in F \wr H$. Necessarily, $\varsigma$ must be $H$-equivariant, i.e.\ $\varphi$ is induced by an $H$-equivariant automorphism of $\bigoplus_H F$. 

\medskip \noindent
Thus, we have proved the decomposition $$\mathrm{Aut}(F \wr H) = \mathrm{Inn}(F \wr H) \cdot \mathrm{Aut}(H) \cdot \mathrm{Aut}_H \left( \bigoplus\limits_H F \right).$$ As $\mathrm{Aut}(H)$ intersects trivially and normalises $\mathrm{Aut}_H\left( \bigoplus_H F \right)$, the desired conclusion follows.
\end{proof}

\noindent
It is worth noticing that the proof above shows that the conclusion of Proposition~\ref{prop:AutOneEnded} holds more generally when $F$ is an arbitrary group and when $H$ is a finitely presented group that does not semisplit over a subgroup of a finite power of $F$. For instance, this applies for wreath products of the form $\mathbb{Z} \wr (\text{rigid hyperbolic group})$. (A hyperbolic group is \emph{rigid} if it is a non-elementary and if it does not split over a virtually cyclic subgroup.)

\medskip \noindent
In view of Proposition~\ref{prop:AutOneEnded}, a natural problem is to investigate the structure of the group of $H$-equivariant automorphisms of $\bigoplus_H F$. In the sequel, we focus on the case when $F$ is cyclic. Even restricting to this case, we shall see that understanding $H$-equivariant automorphisms may be quite tricky.

\begin{description}
	\item[Units of $F {[} H {]}$.] If $F$ is cyclic, identifying $\bigoplus_H F$ with the group ring $F[H]$ allows us to endow $\bigoplus_H F$ with a product $\cdot $ such that $(\bigoplus_H F, + , \cdot )$ is a ring. For every left-invertible $u \in \bigoplus_H F$, the map $$\overline{u} : \left\{ \begin{array}{ccc} F \wr H & \to & F \wr H \\ (c,p) & \mapsto & (c \cdot u, p) \end{array} \right.$$ defines an automorphism of $F \wr H$. Observe that two distinct units induce two distinct automorphisms of $F \wr H$. We denote by $F[H]^\times$ the subgroup of $\mathrm{Aut}(F \wr H)$ consisting of all these automorphisms.
\end{description}
Verifying that $\overline{u}$ defines an automorphism of $F \wr H$ is a straightforward computation once we have observed that the action of $H$ on $\bigoplus_H F$ by coordinate-shifting coincides with the action of $H$ by left multiplication in the ring $F[H]$. This observation also implies that $F[H]^\times \subset \mathrm{Aut}_H(\bigoplus_HF)$. The reverse inclusion turns out to hold as well:

\begin{lemma}\label{lem:Units}
Assume that $F$ is cyclic. The $H$-equivariant automorphisms of $\bigoplus_H F$ are induced by the units of $F[H]$.
\end{lemma}

\begin{proof}
Let $\sigma \in \mathrm{Aut}( \bigoplus_H F)$ be an $H$-equivariant automorphism. Being $H$-equivariant amounts to saying that $\sigma(h\cdot  c) = h \cdot  \sigma(c)$ for all $h \in H$ and $c \in \bigoplus_H F$. Because $\sigma$ is automatically $F$ linear, it follows that $\sigma(c_1 \cdot  c_2) =c_1 \cdot \sigma(c_2)$ for all $c_1,c_2 \in \bigoplus_H F$. In other words, $\sigma$ is a morphism of $(\bigoplus_H F)$-module. Consequently, 
$$\sigma(c)= \sigma(c \cdot  1) = c \cdot  \sigma(1) \text{ for every } c \in \bigoplus_H F,$$
where $\sigma(1)$ is left-invertible (indeed, if $c \in \bigoplus_H F$ is such that $\sigma(c)=1$, then one has $c \cdot  \sigma(1) = \sigma(c) = 1$). 
\end{proof}

\noindent
As a consequence of Lemma~\ref{lem:Units}, understanding the $H$-equivariant automorphisms of $F \wr H$ is as hard as understanding the units of the group ring $F[H]$. The latter problem is well-known to be quite difficult. Indeed, the so-called Kaplansky unit conjecture, which states that the group ring $F[H]$ has only trivial units when $F$ is a field and $H$ torsion-free, remained an open problem for a long time and has been disproved only recently \cite{Kaplansky}. Therefore, Lemma~\ref{lem:Units} suggests that the structure of $\mathrm{Aut}_H(\bigoplus_H F)$, and a fortiori of $\mathrm{Aut}(F \wr H)$, may be highly non-trivial and goes beyond the scope of our current understanding in full generality. 

\medskip \noindent
Before turning to the proof of Corollary~\ref{cor:IntroAuto} from the introduction, let us record the following elementary observation:

\begin{fact}\label{fact:SelfNormalising}
Let $A,B$ be two groups. The normaliser of $B$ in the wreath product $A \wr B$ coincides with $\{ (c,p) \in A \wr B \mid p \in B, c : B \to A \text{ constant} \}$. Consequently, if $B$ is infinite then it is a self-normalising subgroup.
\end{fact}

\begin{proof}
Let $g \in \bigoplus_BA$ be an element normalising $B$, i.e.\ $gBg^{-1}=B$. For every $b \in B$, observe that $gbg^{-1} \cdot b^{-1} \in gBg^{-1} \cdot B=B$ and that
$$g \cdot bg^{-1}b^{-1}  \in \bigoplus_BA \cdot \bigoplus_BA = \bigoplus_BA,$$
hence $gbg^{-1}b^{-1}=1$. In other words, $g$ commutes with all the elements of $B$. Because $B$ acts on $\bigoplus_BA$ by permuting the coordinates, it follows that all the coordinates of $g$ must be identical. When thought of as a map $B \to A$, this amounts to saying that $g$ is constant. 
\end{proof}

\begin{proof}[Proof of Corollary~\ref{cor:IntroAuto}.]
It follows from Proposition~\ref{prop:AutOneEnded} and Lemma~\ref{lem:Units} that
$$\mathrm{Aut}(A \wr B ) = \mathrm{Inn}(A \wr B ) \cdot A [B]^\times \cdot \mathrm{Aut}(B).$$
Because a conjugation by an element of $A \wr B$ decomposes as a product of a conjugation by an element of $\bigoplus_BA$ with a conjugation by an element of $B$, and that the latter also belongs to $\mathrm{Aut}(B)$, we have
$$\mathrm{Aut}(A \wr B ) = \bigoplus\limits_B A \cdot \mathrm{Aut}(B) \cdot A[B]^\times \cdot \mathrm{Aut}(B)= \bigoplus\limits_{B} A \cdot A [B]^\times \cdot \mathrm{Aut}(B),$$
where the second equality follows from the fact that $\mathrm{Aut}(B)$ normalises $A[B]^\times$, which implies $\mathrm{Aut}(B) \cdot A[B]^\times = A[B]^\times \cdot \mathrm{Aut}(B)$. Clearly, $\bigoplus_{B} A$ and $A [B]^\times$ intersect trivially and commute, so $\bigoplus_{B} A \cdot A [B]^\times = \bigoplus_{B} A \oplus A [B]^\times$. Also, observe that $\mathrm{Aut}(B)$ intersects trivially $\bigoplus_{B} A \cdot A [B]^\times$. Indeed, if $\bar{\varphi} = \iota(g) \circ \bar{u}$ for some automorphism $\varphi \in \mathrm{Aut}(B)$, some unit $u \in A[B]^\times$, and some element $g \in \bigoplus_BA$, then
$$\left( c \circ \varphi^{-1}, \varphi(p) \right) = \left( c \cdot u, gpg^{-1} \right) \text{ for all } (c,p) \in A \wr B.$$
Hence $\varphi(p)=gpg^{-1}$ for every $p \in B$. But we deduce from $B= \bar{\varphi}(B)=gBg^{-1}$ and from Fact~\ref{fact:SelfNormalising} that $g \in B$, and in fact that $g=1$ since we already know that $g$ belongs to $\bigoplus_BA$. Therefore, $\varphi$ must be the identity, justifying that $\mathrm{Aut}(B)$ indeed intersects trivially $\bigoplus_{B} A \cdot A [B]^\times$. Moreover, $\mathrm{Aut}(B)$ clearly normalises both $\bigoplus_{B} A$ and $A [B]^\times$. The desired decomposition of $\mathrm{Aut}(A \wr B)$ follows. 
\end{proof}

\noindent
Outside the one-ended case, the elementary automorphisms introduced so far are not sufficient to generate the automorphism group of the corresponding lamplighter group. It seems difficult to exhibit a structure that is both completely general and really informative. Nevertheless, one can push further.

\begin{description}
	\item[Partial conjugations.] Assume that $F$ is an abelian group and that $H$ splits as a free product $H_1 \ast H_2$. Given an $a \in \bigoplus_H F$, the map $$\kappa : s \mapsto \left\{ \begin{array}{cl} s & \text{if $s \in F \cup H_1$} \\ asa^{-1} & \text{if $s \in H_2$} \end{array} \right.$$ induces an automorphism of $F \wr H$.
\end{description} 
Here, $\kappa$ is only defined on generators of $F \wr H$, and it is naturally extended to words of generators. In order to justify that $\kappa$ induces a well-defined morphism $F \wr H \to F \wr H$, it suffices to show that a relator of $F,H_1,H_2$ or from the relative presentation
$$\left\langle F, H \mid \left[ hfh^{-1}, f \right] =1 \text{ where } f \in F, h \in H \backslash \{1\} \right\rangle$$
has always trivial image under $\kappa$. Clearly, $\kappa$ sends a relator of $F,H_1,H_2$ to the identity element, so the desired conclusion follows from the observation that $\kappa$ fixes $\bigoplus_H F$ pointwise. Indeed, let $f \in F$ and $h \in H$ be two elements. Decompose $h$ as an alternating product $a_1b_1 \cdots a_nb_n$ where $a_1,\ldots, a_n \in H_1$ and $b_1, \ldots, b_n \in H_2$. Observe that
$$\begin{array}{lcl} \kappa \left( a_1b_1 \cdot f \cdot b_1^{-1}a_1^{-1} \right) & = &  a_1 ab_1a^{-1} \cdot f \cdot ab_1^{-1} a^{-1}a_1^{-1}\\ \\ & = & a_1 ab_1 \cdot f \cdot b_1^{-1} a^{-1}a_1^{-1} = a_1b_1 \cdot f \cdot b_1^{-1}a_1, \end{array} $$
where the penultimate equality is justified by the fact that $F$ is abelian and where the last equality is justified by the fact that $b_1fb_1^{-1}$ commutes with $a$ since $\bigoplus_HF$ is abelian. By iterating, one shows that $\kappa \left( hfh^{-1} \right) = hfh^{-1}$, as desired. That the kernel of $\kappa$ is trivial is a direct consequence of the facts that $\kappa$ induces the identity on $H$ and restricts to the identity on $\bigoplus_HF$. Finally, the surjectivity of $\kappa$ follows from the easy observation that $F, H_1,H_2$ all lie in the image of $\kappa$. 
\begin{description}
	\item[Transvections.] Assume that $F$ is abelian and that $H$ splits as a free product $H' \ast Z$ where $Z$ is cyclic, say generated by $z \in Z$ of order $r \in \mathbb{N} \cup \{\infty\}$. Given an $a \in \bigoplus_H F$ such that $(az)^r=1$ if $r< \infty$, the map $$\kappa : s \mapsto \left\{ \begin{array}{cl} s & \text{if $s \in F \cup H'$} \\ as & \text{if $s=z$} \end{array} \right.$$ induces an automorphism of $F \wr H$.
\end{description}
Again, $\kappa$ is only defined on generators and is naturally extended to words of generators. In order to justify that $\kappa$ induces a well-defined automorphism of $F \wr H$, it suffices to reproduce the previous argument almost word for word, starting from the relative presentation
$$\left\langle F,H',z \mid z^r=1, \ \left[hfh^{-1},f \right]=1 \text{ where } f \in F, h \in \langle H',z \rangle \backslash\{1\} \right\rangle.$$
When $r$ is finite, the elements $a \in \bigoplus_H F$ satisfying $(az)^r=1$ are obtained in the following way. Decompose $\mathrm{supp}(a)$ as a disjoint union $S_1 \sqcup \cdots \sqcup S_k$ of $\langle z \rangle$-orbits. Write $a$ as $a_1 \cdots a_k$ such that $\mathrm{supp}(a_i)=S_i$ for every $1 \leq i \leq k$. Observe that
$$\begin{array}{lcl} (az)^r & = & a \cdot zaz^{-1} \cdots z^{r-1}az^{-r+1} \\ \\ & = & \left( a_1 \cdot za_1z^{-1} \cdots z^{r-1}a_1z^{-r+1} \right) \cdots \left( a_k \cdot za_kz^{-1} \cdots z^{r-1}a_kz^{-r+1} \right). \end{array}$$ 
Because $z$ has order $r$, each $S_i$ has size $r$ and $z$ acts on it through a cyclic shifting, hence 
$$a_i \cdot za_iz^{-1} \cdots z^{r-1}a_i z^{-r+1} = \prod\limits_{s \in S_i} s.$$
Thus, in order to construct $a$, we took a collection of $\langle z \rangle$-invariant subsets $S_1, \ldots, S_k \subset H$, and, for each $1 \leq i \leq k$, we coloured the lamps in $S_i$ such that the product of all the colours is trivial.

\medskip \noindent
We emphasize that the free product decompositions used in order to define partial conjugations and transvections may be trivial. For instance, $\mathbb{Z} \wr \mathbb{Z}$ has transvections (which turn out to generate the outer automorphism group up to finite index, see Proposition~\ref{prop:AutoCyclic} below).

\medskip \noindent
Thanks to these new families of automorphisms, we are now able to generalise Proposition~\ref{prop:AutOneEnded} to some multi-ended groups.

\begin{prop}\label{prop:AutFreeProduct}
Let $F$ be a finite cyclic group and $H$ a finitely presented group that decomposes as a free product of one-ended and cyclic groups. Then 
$$\mathrm{Aut}(F \wr H) = \left(\mathrm{PConj}(F \wr H) \cdot \mathrm{Trans}(F \wr H) \cdot F[H]^\times\right)\rtimes \mathrm{Aut}(H) .$$
\end{prop}

\noindent
We emphasize that, in this statement, the partial conjugations and the transvections are taken with respect to a fixed free product decomposition. 

\begin{proof}[Proof of Proposition~\ref{prop:AutFreeProduct}.]
Decompose $H$ as a free product $H_1 \ast \cdots \ast H_n$ of one-ended and cyclic groups, and let $\varphi\in \mathrm{Aut}(F \wr H)$ be an automorphism. Because $\bigoplus_H F$ is characteristic, up to composing $\varphi$ with an automorphism of $H$ we assume that $\varphi$ induces the identity on $H$. 

\medskip \noindent
Assume that there exists some $1 \leq i \leq n$ such that $H_i$ is one-ended. As a consequence of Theorem~\ref{thm:Main}, $\varphi(H_i)$ must lie in a conjugate of $H$. So there exist $1 \leq j \leq n$ and $g \in F \wr H$ such that $\varphi(H_i) \subset gH_jg^{-1}$. Since we know that the projection of $\varphi(H_i)$ on $H$ is $H_i$, necessarily $j=i$ and $g \in (\bigoplus_H F) H_i$. Therefore, $\varphi(H_i) \subset a H_i a^{-1}$ for some $a \in \bigoplus_H F$. By composing $\varphi$ with a partial conjugation, we can assume that $\varphi(H_i) \subset H_i$. Because $\varphi$ induces the identity on $H$, it follows that $\varphi$ fixes $H_i$ pointwise.

\medskip \noindent
Next, assume that there exists some $1 \leq i \leq n$ such that $H_i$ is cyclic. Fix a generator $h_i \in H_i$. Because $\varphi$ induces the identity on $H$, there must exist some $a \in \bigoplus_H F$ such that $\varphi(h_i)=ah_i$. Of course, $h_i$ and $ah_i$ have the same order, since $\varphi$ is an automorphism, so we can assume that $\varphi$ fixes $h_i$ up to composing it with a transvection.

\medskip \noindent
Thus, we have proved that, up to composing $\varphi$ with an automorphism of $H$ followed by a product of pairwise commuting partial conjugations and transvections, we can assume that $\varphi$ fixes $H$ pointwise. Then $\varphi$ is induced by an $H$-equivariant automorphism of $\bigoplus_H F$, and the desired conclusion follows from Lemma~\ref{lem:Units}.

\medskip \noindent
In other words, we have proved that $$\mathrm{Aut}(F \wr H) = \mathrm{Aut}(H) \cdot \mathrm{PConj}(F \wr H) \cdot \mathrm{Trans}(F \wr H) \cdot F[H]^\times.$$ Because $\mathrm{Aut}(H)$ normalises $\mathrm{PConj}(F \wr H)$, $\mathrm{Trans}(F \wr H)$, and $F[H]^\times$, and because it intersects trivially $\mathrm{PConj}(F \wr H) \cdot \mathrm{Trans}(F \wr H) \cdot F[H]^\times$, the desired decomposition follows.
\end{proof}

\noindent
Proposition~\ref{prop:AutFreeProduct} describes the automorphism groups of many interesting lamplighter groups, but we are still far away from a global picture. Let us conclude this section with a few open questions. 

\medskip \noindent
First, it is worth noticing that lamplighters over groups that are not finitely presented are not considered at all in the article. This is justified by the fact that admitting a finite presentation is fundamental in Theorem~\ref{thm:Main}. However, it is natural to ask what happens for infinitely presented groups. For instance:

\begin{question}
Let $F$ be a non-trivial finite group. Describe the automorphism group of the iterated wreath product $F \wr ( F \wr ( \cdots ( F \wr \mathbb{Z}) \cdots ))$.
\end{question}

\noindent
In the article, we mainly focused on lamps with only finitely many colours. Even though we saw that the proof of Proposition~\ref{prop:AutOneEnded} also provides interesting information outside this situation, many cases of interest are not covered. As a particular case:

\begin{question}
Describe the automorphism group of $\mathbb{F}_n \wr \mathbb{Z}$, $n \geq 2$.
\end{question}

\noindent
We excluded the case $n=1$ from the question because it can be easily solved, due to the fact that $\mathbb{Z}$ satisfies the unit Kaplansky conjecture over $\mathbb{Z}$. In fact, we can prove more generally that:

\begin{prop}\label{prop:AutoCyclic}
Let $Z$ be a cyclic group (finite or infinite). Then
$$\mathrm{Aut}(Z \wr \mathbb{Z})= Z[\mathbb{Z}]^\times \cdot \mathrm{Inn}(Z \wr \mathbb{Z}) \cdot \{ \mathrm{Trans}(z), z \in Z\} \cdot \langle \iota \rangle$$
where $\iota$ denotes the automorphism that inverts the $\mathbb{Z}$-factor and where $\mathrm{Trans}(z)$ denotes the transvection associated to $z$. Moreover, 
\begin{itemize}
	\item if $Z= \mathbb{Z}$, then $Z[\mathbb{Z}]^\times = \{ \pm X^n \mid n \in \mathbb{Z}\} = \langle X, -1 \rangle \simeq \mathbb{Z}\oplus \mathbb{Z}/2\mathbb{Z}$;
	\item if $Z= \mathbb{Z}/k\mathbb{Z}$ with $k=p_1^{r_1}\cdots p_n^{r_n}$ the prime decomposition of $k$, then
\[Z[\mathbb{Z}]^\times = \left\langle X,Z^\times,1+ p_1\ldots p_n Z \left[X,X^{-1}\right]  \right\rangle.\]
Hence $Z[\mathbb{Z}]^\times$ is finitely generated if and only if $r_1= \cdots = r_n=1$.
\end{itemize}
As a consequence, $\mathrm{Out}(Z \wr \mathbb{Z})$ is virtually infinite cyclic if $Z= \mathbb{Z}$, it is virtually $\mathbb{Z}^{n-1}$ if $Z$ has order a product of $n$ distinct primes, and it contains an infinitely generated abelian subgroup of finite index if the order of $Z$ is divisible by a square.  
\end{prop}

\noindent
In this statement, we identify $Z[\mathbb{Z}]$ with $Z[X,X^{-1}]$. Automorphisms of lamplighter groups $(\text{finite cyclic group}) \wr \mathbb{Z}$ have been already considered in the literature; see for instance \cite{MR918632, MR3438162, MR4308638}. We include a short and elementary proof for completeness.

\begin{proof}[Proof of Proposition~\ref{prop:AutoCyclic}.]
Let $\varphi$ be an automorphism of $Z \wr \mathbb{Z}$. Because $\bigoplus_\mathbb{Z} Z$ is characteristic, $\varphi$ induces an isomorphism $\mathbb{Z} \to \mathbb{Z}$, which we can assume to be trivial up to composing $\varphi$ with $\iota$. Consequently, there exists some $a \in \bigoplus_\mathbb{Z} Z$ such that $\varphi(z)=az$ where $z$ denotes a generator of the $\mathbb{Z}$-factor. Up to composing $\varphi$ with a transvection, we assume that $a=1$. So $\varphi$ fixes pointwise the $\mathbb{Z}$-factor. In other words, $\varphi$ has to be a $\mathbb{Z}$-equivariant automorphism of $\bigoplus_\mathbb{Z} Z$, which amounts to saying that $\varphi$ belongs to $Z[\mathbb{Z}]^\times$ according to Lemma~\ref{lem:Units}. Thus, we have proved that
$$\mathrm{Aut}(Z \wr \mathbb{Z})= Z[\mathbb{Z}]^\times \cdot \mathrm{Trans}(Z \wr \mathbb{Z}) \cdot \langle \iota \rangle.$$
The desired composition of $\mathrm{Aut}(Z \wr \mathbb{Z})$ now follows from the following observation:

\begin{claim}\label{claim:TransInn}
We have $\mathrm{Trans}(Z \wr \mathbb{Z}) \subset \mathrm{Inn}(Z \wr \mathbb{Z} ) \cdot \{ \mathrm{Trans}(z) , z \in Z\}$. 
\end{claim}

\noindent
Let $\rho : \bigoplus_\mathbb{Z} Z \to Z$ denote the morphism that computes the sum of the coordinates. We claim that, for every $g \in \bigoplus_\mathbb{Z} Z$, the transvection $\mathrm{Trans}(g)$ is an inner automorphism if and only if $g \in \mathrm{ker}(\rho)$. This observation implies Claim~\ref{claim:TransInn}. 

\medskip \noindent
Fix a $g \in \bigoplus_\mathbb{Z} Z$ and let $z$ be a generator of the $\mathbb{Z}$-factor of $Z \wr \mathbb{Z}$. Then $\mathrm{Trans}(g)$ is an inner automorphism if and only if there exists $a \in \bigoplus_\mathbb{Z} Z$ such that $g \cdot z = aza^{-1}$, or equivalently $g=aza^{-1}z^{-1}$ since $aza^{-1} = aza^{-1} z^{-1} \cdot z$. In other words, we are trying to find an element $a \in \bigoplus_\mathbb{Z} Z$ such that, by adding $a$ to the $z$-shift of by its inverse, we obtain $g$. There is only one possibility: the leftmost coordinate of $a$ has to coincide with the leftmost coordinate of $g$, say $a_i=g_i$; next, $a_{i+1}$ has to be $g_{i+1}+g_i$; and $a_{i+2}$ has to be $g_{i+2}+g_{i+1}+g_i$, and so on. A priori, the element $a$ thus defined belong to $\prod_{\mathbb{Z}} Z$. It belongs to $\bigoplus_{\mathbb{Z}} Z$ if and only if $\rho(g)=0$. This concludes the proof of Claim~\ref{claim:TransInn}. 

\medskip \noindent
Now, we focus on the structure of $Z[\mathbb{Z}]^\times$, which we identify with $Z[X,X^{-1}]^\times$. The first item of our statement is a direct consequence of the following observation:

\begin{claim}\label{claim:UnitsIntegral}
If $R$ is an integral domain, then 
$$R\left[ X,X^{-1} \right]^\times= \{ r X^n \mid r \in R^\times, n \in \mathbb{Z} \} = \langle R^\times, X \rangle \simeq R^\times \oplus \mathbb{Z}.$$
\end{claim}

\noindent
Finding the units in $R[X,X^{-1}]$ amounts to finding the divisors of the $X^n$, $n\in \mathbb{N}$, in $R[X]$. But, because $R$ is an integral domain, the product of two polynomials such that at least of them is not monomial cannot be monomial since the products of the terms of highest and lowest degrees provide at least two terms. This proves Claim~\ref{claim:UnitsIntegral}. 

\medskip \noindent
Next, let us proves the second item of our statement. Clearly $X$ and elements of $Z^\times$ are invertible. On the other hand, since $p_1\ldots p_n$ is nilpotent, $J=p_1\ldots p_n Z \left[X,X^{-1}\right]$ is contained in the nilradical of $Z \left[X,X^{-1}\right]$, and therefore $1+a$ is invertible for all $a\in J$. 
Conversely, let $b\in Z \left[X,X^{-1}\right]^\times$. Reducing modulo $p_1\ldots p_k$, we see that at least one coefficient is invertible. On multiplying it by a power of $X$ and an element of $Z^\times$, we may assume that 
$b=1+a$, where the constant coefficient of $a$ equals $0$. Reducing modulo $p_i$ and applying Claim \ref{claim:UnitsIntegral} with $R=\mathbb{Z}/p_i\mathbb{Z} $, we deduce that $p_i$ divide $a$. Hence $p_1\ldots p_n$ divides $a$, and we are done.

\medskip \noindent Let us check that if some $r_i\geq 2$, then $Z \left[X,X^{-1}\right]^\times$ is not finitely generated. Denote $p=p_i$. Projecting to  $\mathbb{Z}/p_i^2\mathbb{Z}\left[X,X^{-1}\right]$, we are left to treating the case where $Z=\mathbb{Z}/p^2\mathbb{Z}$. Denote $A=Z \left[X,X^{-1}\right]$. 
Note that for all $x,y\in A$, $(1+px)(1+py)=1+p(x+y)$. This implies that $A^\times\cap A_m=Z^\times+pA_m$ is a subgroup of $A^\times$, where  $A_m$ is the set of polynomials of degree at most $m$ in $X,X^{-1}$. Hence, if $A^\times$ was finitely generated, then the degree of its elements would be bounded: contradiction.

\medskip \noindent
Finally, the description of $\mathrm{Out}(Z\wr \mathbb{Z})$ follows from the description of $\mathrm{Aut}(Z\wr \mathbb{Z})$ together with the observation that $Z[\mathbb{Z}]^\times \cap \mathrm{Inn}(Z \wr \mathbb{Z}) = \langle X \rangle$. 
\end{proof}

\noindent
For wreath products of the form $(\text{finite group})\wr ( \text{multi-ended group})$, we only considered the case where the finite group is abelian, because otherwise the partial conjugations and transvections may not be well-defined. This suggests the following problem.

\begin{question}
Let $F$ be a non-abelian finite group and $n \geq 2$ an integer. Describe the automorphism group of $F \wr \mathbb{F}_n$.
\end{question}

\noindent
The difficulty in extending Proposition~\ref{prop:AutFreeProduct} to lamplighters over arbitrary multi-ended groups comes from the torsion, which may introduce sufficient flexibility in order to define new types of automorphisms.
\begin{description}
	\item[Generalised partial conjugations.] Assume that $F$ is abelian and that $H$ splits as a graph of groups with underlying graph $\Gamma$. Fix a marked vertex $v_0 \in V(\Gamma)$, an orientation on $\Gamma$, and, for every edge $e \in E(\Gamma)$, an element $a_e \in C(H_e) \cap \bigoplus_H F$. Every path in $\Gamma$ is naturally labelled by the element obtained by reading the oriented edges it passes through. Assume that the label of every cycle is trivial. For every vertex $v \in V(\Gamma)$, let $a_v \in \bigoplus_H F$ denote the label of a path from $v_0$ to $v$. (As a consequence of our previous assumption, this label does not depend on the path we chose.) Fix a maximal subtree $T \subset \Gamma$, let $e_1, \ldots, e_k$ denote the oriented edges outside $T$, and let $t_1, \ldots, t_k$ denote the corresponding stable letters. The map
$$\kappa : s \mapsto \left\{ \begin{array}{cl} s & \text{if $s \in F$} \\ a_vsa_v^{-1} & \text{if $s \in H_v$} \\ a_{o(e_i)}t_i a_{o(e_i)}^{-1} & \text{if $s=t_i$} \end{array} \right.$$ induces an automorphism of $F \wr H$. 
\end{description}
The proof is in the same spirit as before, but we use Bass-Serre theory in order to find a relative presentation of $H$. We do not include the details here. Similarly, one can define generalised transvections.

\medskip \noindent
Understanding such automorphisms in full generality turns out to be tricky since they seem to depend on the structures of the finite groups involved and on how they interact together. As a test question, we propose:

\begin{question}
Let $A,B$ be two non-trivial finite groups and $C \leq B$ a subgroup. Describe the automorphism group of $A \wr (B \ast_C B)$. 
\end{question}

\appendix

\section{A word about approximating infinitely presented groups}

\noindent
The main idea underlying the article is that a group $H$ admitting a presentation of the form $\langle X \mid R,r_1,r_2, \ldots \rangle$, where $R$ is a set of relations and where $r_1,r_2,\ldots$ is an infinite sequence of relations, can be ``approximated'' by the groups $H_i$ defined by the truncated presentations $\langle X \mid R, r_1, \ldots, r_i \rangle$. The proof of Theorem~\ref{thm:Main} is an algebraic illustration of this idea applied to wreath products. This appendix is a general discussion dedicated to this strategy from different points of view and for other infinitely presented groups.

\paragraph{Algebraic approximation.} The fact that the $H_i$ algebraically approximate, in a sense, the group $H$ is illustrated by the following (well-known) observation:

\begin{fact}\label{fact:Algebra}
For every finitely presented group $G$ and every morphism $\varphi : G \to H$, there exist an $i \geq 1$ and a morphism $\psi : G \to H_i$ such that $\varphi = \pi_i \circ \psi$. 
\end{fact}

\noindent
The proof follows the same lines as Lemma~\ref{lem:Approx}, which is a particular case. See for instance \cite[Lemma~3.1]{MR2764930}. Interestingly, the $H_i$ may satisfy properties that are quite different from the properties satisfied by $H$. For instance, applying the approximation given in Section~\ref{section:Proof} to the lamplighter group $H=\mathbb{Z}/2\mathbb{Z} \wr \mathbb{Z}$ yields virtually non-abelian free groups $H_i$, even though $H$ is solvable. This implies that every finitely presented group admitting $\mathbb{Z}/2\mathbb{Z} \wr \mathbb{Z}$ as a quotient must virtually surject onto $\mathbb{F}_2$, as already mentioned in the introduction. Similar statements can be obtained for infinitely presented groups that are not wreath products, as shown below.

\paragraph{Geometric approximation.} Fact~\ref{fact:Algebra} admits a geometric version, justifying that the $H_i$ also provide geometric approximations of $H$. More precisely:

\begin{fact}\label{fact:Geometry}
Let $X$ be a coarsely $1$-connected graph. For every coarse (resp. quasi-isometric) embedding $\rho : X \to H$, there exist an $i \geq 1$ and a coarse (resp. quasi-isometric) embedding $\eta : X \to H_i$ such that $\rho = \pi_i \circ \eta$. 
\end{fact}

\noindent
Recall that a graph $X$ is \emph{coarsely $1$-connected} if there exists an $R \geq 0$ such that the $2$-complex obtained from $X$ by filling in with discs all the cycle of length $\leq R$ is simply connected. If $X$ is a Cayley graph of a finitely generated group $G$ (constructed from a finite generating set), then $X$ is coarsely $1$-connected precisely when $G$ is finitely presented. 

\medskip \noindent
Fact~\ref{fact:Geometry} is central in the proof of the embedding theorem obtained in \cite[Theorem~5.1]{GTwreath}; see \cite[Lemma~2.2]{GTwreath}.

\paragraph{Logic approximation.} One can also observe that, in the space of marked groups, the $(H_n,X)$ converge to $(H,X)$. Because convergence in the space of marked groups is connected to first-order properties, one obtains the following observation:

\begin{fact}\label{fact:Logic}
Let $G$ be a finitely presented group and $S \subset G$ a finite generating set. If $\mathrm{Th}_\forall(G) \supset \mathrm{Th}_\forall (H)$, then $(G,S)$ is a limit of marked groups $(K_n,S_n)$ where each $K_n$ is both a subgroup of the truncation $H_n$ and a quotient of $G$.
\end{fact}

\noindent
We refer to \cite{MR2151593} for details about marked groups and universal theories of groups. Nevertheless, let us recall that, given a group $G$, its \emph{universal theory} $\mathrm{Th}_\forall (G)$ refers to the set of all the universal first-order formulae, i.e.\ the sentences of the form 
$$\forall x_1, \ldots, x_n \in G, \bigvee\limits_{i=1}^p \left( \bigwedge\limits_{j =1}^q w_{ij}( x_1, \ldots, x_n) = 1 \wedge \bigwedge_{k=1}^r w_{ik}( x_1, \ldots, x_n) \neq 1 \right)$$ 
with the $w_{ij},w_{ik}$ words written over $\{x_1, \ldots, x_n\}^{\pm 1}$, that are satisfied by $G$. 

\begin{proof}[Proof of Fact~\ref{fact:Logic}.]
According to \cite[Proposition~5.3]{MR2151593}, the inclusion $\mathrm{Th}_\forall(G) \supset \mathrm{Th}_\forall (H)$ implies that $(G,S)$ is a limit of marked subgroups of $H$ in the space of marked groups. But we also know that $H$ (endowed with an arbitrary marking) is the limit of the $H_i$ (for some marking). Consequently, $(G,S)$ is a limit of marked subgroups $K_n \leq H_n$. Since we know that some neighbourhood of $(G,S)$ contains only marked quotients of $G$ \cite[Lemma~2.3]{MR2151593}, we can suppose without loss of generality that $H_n$ is also a quotient of $G$ for every $n \geq 0$. 
\end{proof}

\noindent
As an illustration, let us show that no finitely presented group have the same universal theory as, for instance, the lamplighter group $\mathbb{Z} \wr \mathbb{Z}$. 

\begin{prop}\label{prop:Logic}
Let $A,B$ be two free abelian groups of finite rank $\geq 1$, and let $G$ be a finitely generated group. If $G$ and $A \wr B$ have the same universal theory, then $G$ is not finitely presented.
\end{prop}

\begin{proof}
Notice that like $A \wr B$, $G$ is metabelian but not abelian, since these properties can be expressed as universal or existential formulae. According to Fact~\ref{fact:Logic}, if $G$ is finitely presented then $(G,S)$ can be described as a limit of marked groups $(K_n,S_n)$, where $S \subset G$ is an arbitrary finite generating set and where each $K_n$ is both a quotient of $G$ and a subgroup of $A \square_{R_n} B$ for some $R_n \subset B$. Up to ignoring finitely many $K_n$, we assume without loss of generality that $K_n$ is metabelian but non-abelian for every $n$. 

\medskip \noindent
Observe that, if the projection of $K_n$ on $B$ is trivial for some $n$, then $K_n$ defines a non-abelian subgroup in the right-angled Artin group $\Gamma_{R_n}A$, which implies that $K_n$ must contain a non-abelian free subgroup. But this is impossible since $K_n$ is metabelian. Thus, the projection of $K_n$ on $B$ must be infinite for every $n$. 

\medskip \noindent
Next, observe that, for every $n$, the intersection $K_n \cap \Gamma_{R_n}A$ is non-trivial. Otherwise, $K_n$ would inject into $B$ and would be abelian.

\medskip \noindent
Finally, fix some $n$ and fix an enumeration $\{g_1,g_2, \ldots \}$ of $K_n \cap \Gamma_{R_n}A$. For every $i \geq 1$, set $K_n^{\leq i} := \langle g_1, \ldots, g_i \rangle$ and let $\Lambda_i$ denote the smallest subgroup in $\Gamma_{R_n}$ such that $K_n^{\leq i}$ lies in a conjugate of $\langle \Lambda_i \rangle \leq \Gamma_{R_n}$ (such a subgraph exists according to \cite[Proposition~3.10]{MR3365774} and it is finite). Because $B$ acts metrically properly on $\Gamma_{R_n}$ and because $K_n$ has an infinite projection on $B$, the sequence $\Lambda_1 \subset \Lambda_2 \subset \cdots$ does not stabilise. This implies that $K_n \cap \Gamma_{R_n}A$ is not finitely generated. Also, since $\Gamma_S$ is locally finite, there must exist an index $i \geq1$ such that $\Lambda_i$ contains at least two vertices and is not a join. We deduce from \cite[Theorem~4.1]{MR3365774} that $K_n \cap \Gamma_{R_n}A$ has to contain a non-abelian free subgroup, a contradiction. 
\end{proof}

\paragraph{Other infinitely presented groups.} At first glance, it might be surprising that a finitely presented group that surjects onto a wreath product $(\text{non-trivial})\wr (\text{infinite})$ must be SQ-universal, even though the wreath product itself may be very far from being SQ-universal (e.g. it may be solvable). As described in Section~\ref{section:Proof}, the reason comes from the fact that wreath products can be approximated by groups satisfying properties of negative curvature, and the SQ-universality is deduced from such properties. 

\medskip \noindent
Interestingly, such a phenomenon is not specific to wreath products and seems to be rather common. For instance, \cite{MR3061134} exhibits many similar examples among metabelian groups and branch groups. In the particular case of the Grigorchuk group $\mathfrak{G}$, the structure of the approximating groups is completely explicit. Recall from \cite{MR819415, MR767207} that $\mathfrak{G}$ admits as a presentation
$$\langle a,b,c,d \mid a^2=b^2=c^2=d^2=bcd=1, u_n=v_n=1 \ (n \geq 0) \rangle$$
where the words $u_n,v_n$ are defined inductively as
$$\left\{ \begin{array}{l} u_0=(ad)^4  \text{ and } v_0=(adacac)^4 \\ u_n=\sigma^n(u_0) \text{ and } v_n=\sigma^n(v_0) \text{ for all } n \geq 0 \end{array} \right.$$
from the substitution $\sigma$ defined by 
$$\sigma(a)=aca, \ \sigma(b)=d, \ \sigma(c)=b, \ \sigma(d)=c.$$
Then, for every $n \geq 0$, the group $\mathfrak{G}_n$ defined by the truncated presentation
$$\left\langle a,b,c,d \left| \begin{array}{l} a^2=b^2=c^2=d^2=bcd=1 \\ u_0= \cdots =u_n=v_0= \cdots =v_{n-1}=1 \end{array} \right. \right\rangle$$
turns out to be virtually a product of free groups (see \cite[Theorem~5.3]{MR3061134} for a more precise statement). Let us mention two applications of this point of view:

\begin{prop}
For every finitely presented group $G$ and every morphism $\varphi : G \to \mathfrak{G}$, if the image of $\varphi$ is not virtually abelian, then $G$ has to be large.
\end{prop}

\begin{proof}
As a consequence of Fact~\ref{fact:Algebra}, $\varphi$ factors through a morphism $\psi : G \to \mathfrak{G}_n$ for some $n \geq 0$. Because $\mathfrak{G}_n$ is virtually a product of free groups, $\psi(G)$ is either virtually abelian or large. In the former case, the image of $\varphi$ in $\mathfrak{G}$ has to be virtually abelian; and in the latter case, $G$ has to be large. 
\end{proof}

\begin{prop}
A finitely generated nilpotent group quasi-isometrically embeds into the Grigorchuk group $\mathfrak{G}$ if and only if it is virtually abelian.
\end{prop}

\begin{proof}
Because a finitely generated nilpotent group is automatically finitely presented, it follows from Fact~\ref{fact:Geometry} that such a group that quasi-isometrically embeds in the Grigorchuk group must quasi-isometrically embed in a product of finitely many free groups. On the other hand, a finitely generated nilpotent group can be realised as a uniform lattice in a connected simply-connected nilpotent Lie group \cite{MR0028843}; and such nilpotent Lie group does not quasi-isometrically embed in a CAT(0) space such as a product of trees according to \cite{MR1883722}, unless it is abelian. We conclude that a finitely generated nilpotent group that quasi-isometrically embeds in the Grigorchuk group must be virtually abelian. Conversely, because the Grigorchuk is commensurable to its own square, $\mathbb{Z}^n$ quasi-isometrically embeds into $\mathfrak{G}$ for every $n \geq 1$, so every finitely generated virtually abelian group quasi-isometrically embeds in the Grigorchuk group. 
\end{proof}

\addcontentsline{toc}{section}{References}

\bibliographystyle{alpha}
{\footnotesize\bibliography{MorphismWreath}}

\Address

\end{document}